\documentclass[12pt]{article}
\usepackage{nicefrac}
\usepackage{epsfig}
\usepackage{algorithm}
\usepackage{algorithmic}
\usepackage{latexsym}
\usepackage{amsmath}
\usepackage{amsfonts}
\usepackage{amsthm}
\usepackage{graphicx}
\usepackage[normalem]{ulem}
\usepackage{enumitem}
\usepackage{array}\usepackage{relsize}
\usepackage{color}\usepackage[dvipsnames]{xcolor}
\usepackage{tikz}\usetikzlibrary{patterns}
\usepackage{authblk}

\def\a{0cm} \def\A{0.5cm}
\def\b{1cm} \def\B{1.5cm}
\def\c{2cm} \def\C{2.5cm}
\def\d{3cm} \def\D{3.5cm}
\def\e{4cm} \def\E{4.5cm}
\def\f{5cm} \def\F{5.5cm}
\def\g{6cm} \def\G{6.5cm}
\def\h{7cm} \def\H{7.5cm}
\def\i{8cm} \def\I{8.5cm}
 
 \def\K{10.5cm}
 
 \def\M{12.5cm}
 
 \def\O{14.5cm}\def\Q{16.5cm}\def\S{18.5cm}
\def\U{20.5cm}\def\W{22.5cm}\def\Y{24.5cm}\def\ZZ{26.5cm}

 \def\ZZ{26.5cm}\def\Za{28.5cm}\def\Zb{30.5cm}\def\Zc{32.5cm}
\def\Zi{50.5cm}

			\def\DUD{\begin{tikzpicture}[scale=0.4]
				\draw [\styleGrille] (\A,\A) -- (\A,\B);
				\draw [\styleGrille] (\B,\A) -- (\B,\B);
				\draw [\styleGrille] (\A,\A) -- (\B,\A);
				\draw [\styleGrille] (\A,\B) -- (\B,\B);
				\draw [\styleGrille] (\A,\A) -- (\B,\B);
				
                \pointt{\A}{\A} 
				\pointt{\B}{\B} 
				\draw (\A,\B) node {$\times$};\draw (\B,\A) node {$\times$};

			\end{tikzpicture}}

\def\sizePoint{3pt}
\def\sizePointt{6pt}

\newcommand{\point}[2]{\fill (canvas cs:x=#1,y=#2) circle (\sizePoint); }
\newcommand{\pointt}[2]{\fill (canvas cs:x=#1,y=#2) circle (\sizePointt); }

\DeclareMathOperator{\Seq}{\textsc{Seq}}
\DeclareMathOperator{\Seqplus}{\Seq^{+}}

\newtheorem{thm}{Theorem}
\newtheorem{cor}[thm]{Corollary}

\title{Bijections between  directed animals, multisets and Grand-Dyck paths}
\author[1]{Jean-Luc Baril\thanks{Corresponding author: \texttt{barjl@u-bourgogne.fr}}}
\author[2]{David Bevan}
\author[1]{Sergey Kirgizov}

\affil[1]{LIB, Univ. Bourgogne Franche-Comt{\'e}, France}
\affil[2]{University of Strathclyde, Glasgow, Scotland, United Kingdom}

\def\styleGrille{densely dotted}

\begin{document}
\maketitle

\begin{abstract} An $n$-multiset of $[k]=\{1,2,\ldots, k\}$ consists of a set of $n$ elements from $[k]$ where each element can be repeated. We present the bivariate
generating function for $n$-multisets of $[k]$ with no consecutive elements. For $n=k$, these multisets have the same enumeration as directed
animals in the square lattice. Then we give constructive bijections between directed animals, multisets with no consecutive elements and Grand-Dyck paths avoiding the pattern $DUD$, and we show how  classical and novel statistics are transported by these bijections.
\end{abstract}



\section{Introduction and motivation}

For several decades, directed animals have been widely studied in the literature.
 They are special lattice point configurations, and their close links with certain problems of thermodynamics, in particular with the problem of directed percolation \cite{Broa,CoGu},  gives them an important place in the domains of theoretical physics and combinatorics. The problem of the enumeration of animals of a given area was stated for the first time by Harary in \cite{Har}. Then, Dhar, Phani and Barma \cite{Dhar} provided a first closed form for counting directed animals on the square and triangular lattices with respect to area. Later other proofs of this result were given using new combinatorial structures such as heaps of pieces \cite{Vie} or gas models \cite{Alb,Bou1,Dha,Leb}. In \cite{Bach} Bacher provides a generating function for the total site perimeter on the square and triangular lattices, solving a conjecture proposed in 1996 by Conway \cite{Conw}. But the problem of finding the generating function for the enumeration of directed animals according to the area and perimeter still remains open. According to \cite{Marck}, this function is not believed to be D-finite.  Hoping to capture properties on the perimeter, other studies present one-to-one correspondences between directed animals and some restricted classes of already known combinatorial objects such as {\it guingois} trees \cite{Betr}, heaps of dimers \cite{Bou}, forests of 1-2 trees \cite{Barc}, permutations avoiding the patterns $321$ and $4\bar{1}523$ \cite{Barc}, and lattice paths \cite{Bousq,Gou}.

 The purpose of this paper is to present new combinatorial classes in bijection with directed animals on the square and triangular lattices. After preliminary discussions on multisets in Section 2, we exhibit in Section 3 a correspondence with Grand-Dyck paths avoiding the pattern $DUD$ that is in turn in correspondence with multisets with no consecutive elements. Then, we show how these bijections transport classical and novel parameters on these classes, which opens a new way to explore statistics on  directed animals in the  well-known contexts of Dyck paths and multisets.

\section{Preliminaries}
 An $n$-{\it multiset} on $[k]=\{1,2,\ldots, k\}$ consists of a set of $n$ elements from  $[k]$ where we permit each element to be repeated \cite{Knu,Stan}.
  Throughout this paper, an $n$-multiset $\pi$ will be represented by the unique  sequence $\pi_1\pi_2\ldots \pi_n$ of its elements ordered in non-decreasing order,
  {\it e.g.} the multiset $\{1,1,2,2,3,3\}$ will be written $112233$. For $n,k\geq 1$, let $\mathcal{M}_{n,k}$  be the set of $n$-multisets of $[k]$.
  We set  $\mathcal{M}_n=\mathcal{M}_{n,n}$ and $\mathcal{M}=\bigcup_{n\geq 1}\mathcal{M}_n$.
  For instance, we have $\mathcal{M}_{3,2}=\{111,112,122,222\}$ and $\mathcal{M}_{3}=\mathcal{M}_{3,2}\cup\{113,123,133,223,233,333\}$.
   The {\it graphical representation} of a multiset $\pi\in \mathcal{M}_{n,k}$ is the set of points in the plane at coordinates $(i,\pi_i)$ for $i\in [n]$.
    Whenever none of the points $(i,\pi_i)$ lie below the diagonal $y=x$, {\it i.e.},  $i\leq \pi_i$ for all $i\in [n]$, $\pi$ will be called
    {\it superdiagonal}. Let $\mathcal{M}^s_{n,k}$ (resp.  $\mathcal{M}^s_{n}$) be the set of superdiagonal $n$-multisets of $[k]$ (resp. of $[n]$) and $\mathcal{M}^s=\bigcup_{n\geq 1}\mathcal{M}^s_n$.

From a multiset $\pi\in \mathcal{M}_{n,k}$, we consider the path of length $n+k$ on its graphical representation  with up and right moves along the edges of the squares that goes from the lower-left corner $(0,0)$
to the upper-right corner  $(n,k)$ and  leaving all the points $(i,\pi_i)$, $i\in[n]$, to the right and remaining always as
close to the line $x=n$ as possible (see the left part of Figure~\ref{fig1} for an example). Then, the number of $n$-multisets of $[k]$ is the number of possibilities to choose $n$ right moves among $k+n-1$ moves (the first is necessarily an up move), that is the binomial coefficient $\binom{k+n-1}{n}$ (see for instance \cite{Stan}).
Reading this path from left to right, we construct a lattice path of length $n+k$ from $(0,0)$ to $(n+k,n-k)$ by replacing any up-move with up step $U=(1,1)$ and
any right-move with down step $D=(1,-1)$. Clearly, this path starts with $U$ and consists of $n$ up steps and $k$ down steps. As a byproduct whenever $n=k$,
this construction induces a bijection  $\Phi$ between $\mathcal{M}_n$  and the set $\mathcal{GD}_n$  of Grand-Dyck paths of semilength $n$ starting with an up-step,
that is the set of paths from $(0,0)$ to $(2n,0)$ starting with $U$ and consisting of $U$ and $D$ steps. Moreover, the image by $\Phi$ of  $\mathcal{M}^s_n$ is
 the set $\mathcal{D}_n$ of Dyck paths of semilength $n$, {\it i.e.} the subset of paths in $\mathcal{GD}_n$ that do not cross the $x$-axis.
See Figure~\ref{fig1} for two examples of this construction.
\begin{thm}\label{thmMGD}
The map $\Phi$ is a bijection from $\mathcal{M}_n$ to $\mathcal{GD}_n$, and the image  of $\mathcal{M}^{s}_n$ is $\mathcal{D}_n$.
\end{thm}

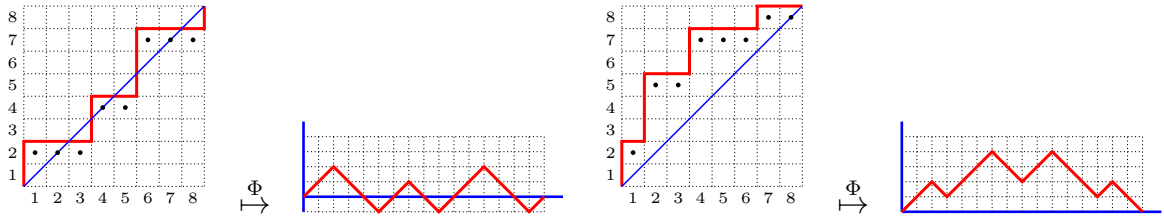
\begin{figure}[ht]
	
		\begin{center}
			\begin{tikzpicture}[scale=0.3]
				\draw [\styleGrille] (\A,\A) -- (\A,\I); \draw (\b,\a) node {\tiny 1};
				\draw [\styleGrille] (\B,\A) -- (\B,\I); \draw (\c,\a) node {\tiny 2};
				\draw [\styleGrille] (\C,\A) -- (\C,\I); \draw (\d,\a) node {\tiny 3};
				\draw [\styleGrille] (\D,\A) -- (\D,\I); \draw (\e,\a) node {\tiny 4};
				\draw [\styleGrille] (\E,\A) -- (\E,\I); \draw (\f,\a) node {\tiny 5};
				\draw [\styleGrille] (\F,\A) -- (\F,\I); \draw (\g,\a) node {\tiny 6};
				\draw [\styleGrille] (\G,\A) -- (\G,\I); \draw (\h,\a) node {\tiny 7};
				\draw [\styleGrille] (\H,\A) -- (\H,\I); \draw (\i,\a) node {\tiny 8};
				\draw [\styleGrille] (\I,\A) -- (\I,\I);
				\draw [\styleGrille] (\A,\A) -- (\I,\A); \draw (\a,\b) node {\tiny 1};
				\draw [\styleGrille] (\A,\B) -- (\I,\B); \draw (\a,\c) node {\tiny 2};
				\draw [\styleGrille] (\A,\C) -- (\I,\C); \draw (\a,\d) node {\tiny 3};
				\draw [\styleGrille] (\A,\D) -- (\I,\D); \draw (\a,\e) node {\tiny 4};
				\draw [\styleGrille] (\A,\E) -- (\I,\E); \draw (\a,\f) node {\tiny 5};
				\draw [\styleGrille] (\A,\F) -- (\I,\F); \draw (\a,\g) node {\tiny 6};
				\draw [\styleGrille] (\A,\G) -- (\I,\G); \draw (\a,\h) node {\tiny 7};
				\draw [\styleGrille] (\A,\H) -- (\I,\H); \draw (\a,\i) node {\tiny 8};
				\draw [\styleGrille] (\A,\I) -- (\I,\I);
				\draw [solid,line width=0.4mm,color=red] (\A,\A) -- (\A,\C) -- (\D,\C) -- (\D,\E) -- (\F,\E) -- (\F,\H) -- (\I,\H)  -- (\I,\I);
\draw [solid, line width=0.2mm, color=blue] (\A,\A) -- (\I,\I);
				\point{\b}{\c} 
				\point{\c}{\c} 
				\point{\d}{\c} 
				\point{\e}{\e} 
				\point{\f}{\e} 
				\point{\g}{\h} 
				\point{\h}{\h} 
				\point{\i}{\h} 
			\end{tikzpicture}
~~$\stackrel{\Phi}{\mapsto}$~~
\begin{tikzpicture}[scale=0.1]
            \draw[\styleGrille] (\A,\A)-- (\Zc,\A);
             \draw[\styleGrille] (\A,\E)-- (\Zc,\E);
              \draw[solid,line width=0.35mm,color=blue](\A,\C)-- (35,\C);
               \draw[\styleGrille] (\A,\G)-- (\Zc,\G);
               \draw[\styleGrille] (\A,\I)-- (\Zc,\I);
              \draw[\styleGrille] (\A,\K)-- (\Zc,\K);
            \draw[solid,line width=0.35mm,color=blue] (\A,\A) -- (\A,\M);
             \draw[\styleGrille] (\C,\A) -- (\C,\K);\draw[\styleGrille] (\E,\A) -- (\E,\K);\draw[\styleGrille] (\G,\A) -- (\G,\K);
             \draw[\styleGrille] (\I,\A) -- (\I,\K);\draw[\styleGrille] (\K,\A) -- (\K,\K);\draw[\styleGrille] (\M,\A) -- (\M,\K);
             \draw[\styleGrille] (\O,\A) -- (\O,\K);\draw[\styleGrille] (\Q,\A) -- (\Q,\K);\draw[\styleGrille] (\S,\A) -- (\S,\K);
             \draw[\styleGrille] (\U,\A) -- (\U,\K);\draw[\styleGrille] (\W,\A) -- (\W,\K);\draw[\styleGrille] (\Y,\A) -- (\Y,\K);

             \draw[\styleGrille] (\ZZ,\A) -- (\ZZ,\K);
             \draw[\styleGrille] (\Za,\A) -- (\Za,\K);
             \draw[\styleGrille] (\Zb,\A) -- (\Zb,\K);
             \draw[\styleGrille] (\Zc,\A) -- (\Zc,\K);

            \draw[solid,line width=0.4mm,color=red] (\A,\C)--(\E,\G) -- (\K,\A) -- (\O,\E) -- (\S,\A) -- (\Y,\G)  -- (\Zb,\A) -- (\Zc,\C);
         \end{tikzpicture}
         ~
			\begin{tikzpicture}[scale=0.3]
				\draw [\styleGrille] (\A,\A) -- (\A,\I); \draw (\b,\a) node {\tiny 1};
				\draw [\styleGrille] (\B,\A) -- (\B,\I); \draw (\c,\a) node {\tiny 2};
				\draw [\styleGrille] (\C,\A) -- (\C,\I); \draw (\d,\a) node {\tiny 3};
				\draw [\styleGrille] (\D,\A) -- (\D,\I); \draw (\e,\a) node {\tiny 4};
				\draw [\styleGrille] (\E,\A) -- (\E,\I); \draw (\f,\a) node {\tiny 5};
				\draw [\styleGrille] (\F,\A) -- (\F,\I); \draw (\g,\a) node {\tiny 6};
				\draw [\styleGrille] (\G,\A) -- (\G,\I); \draw (\h,\a) node {\tiny 7};
				\draw [\styleGrille] (\H,\A) -- (\H,\I); \draw (\i,\a) node {\tiny 8};
				\draw [\styleGrille] (\I,\A) -- (\I,\I);
				\draw [\styleGrille] (\A,\A) -- (\I,\A); \draw (\a,\b) node {\tiny 1};
				\draw [\styleGrille] (\A,\B) -- (\I,\B); \draw (\a,\c) node {\tiny 2};
				\draw [\styleGrille] (\A,\C) -- (\I,\C); \draw (\a,\d) node {\tiny 3};
				\draw [\styleGrille] (\A,\D) -- (\I,\D); \draw (\a,\e) node {\tiny 4};
				\draw [\styleGrille] (\A,\E) -- (\I,\E); \draw (\a,\f) node {\tiny 5};
				\draw [\styleGrille] (\A,\F) -- (\I,\F); \draw (\a,\g) node {\tiny 6};
				\draw [\styleGrille] (\A,\G) -- (\I,\G); \draw (\a,\h) node {\tiny 7};
				\draw [\styleGrille] (\A,\H) -- (\I,\H); \draw (\a,\i) node {\tiny 8};
				\draw [\styleGrille] (\A,\I) -- (\I,\I);
				\draw [solid,line width=0.4mm,color=red] (\A,\A) -- (\A,\C) -- (\B,\C) -- (\B,\F) -- (\D,\F) -- (\D,\H) -- (\G,\H) -- (\G,\I) -- (\I,\I);
\draw [solid, line width=0.2mm, color=blue] (\A,\A) -- (\I,\I);
				\point{\b}{\c} 
				\point{\c}{\f} 
				\point{\d}{\f} 
				\point{\e}{\h} 
				\point{\f}{\h} 
				\point{\g}{\h} 
				\point{\h}{\i} 
				\point{\i}{\i} 
			\end{tikzpicture}
~~$\stackrel{\Phi}{\mapsto}$~~
\begin{tikzpicture}[scale=0.1]
            \draw[solid,line width=0.35mm,color=blue] (\A,\A)-- (35,\A);
             \draw[\styleGrille] (\A,\E)-- (\Zc,\E);
              \draw[\styleGrille] (\A,\C)-- (\Zc,\C);
               \draw[\styleGrille] (\A,\G)-- (\Zc,\G);
               \draw[\styleGrille] (\A,\I)-- (\Zc,\I);
              \draw[\styleGrille] (\A,\K)-- (\Zc,\K);

            \draw[solid,line width=0.35mm,color=blue] (\A,\A) -- (\A,\M);
             \draw[\styleGrille] (\C,\A) -- (\C,\K);\draw[\styleGrille] (\E,\A) -- (\E,\K);\draw[\styleGrille] (\G,\A) -- (\G,\K);
             \draw[\styleGrille] (\I,\A) -- (\I,\K);\draw[\styleGrille] (\K,\A) -- (\K,\K);\draw[\styleGrille] (\M,\A) -- (\M,\K);
             \draw[\styleGrille] (\O,\A) -- (\O,\K);\draw[\styleGrille] (\Q,\A) -- (\Q,\K);\draw[\styleGrille] (\S,\A) -- (\S,\K);
             \draw[\styleGrille] (\U,\A) -- (\U,\K);\draw[\styleGrille] (\W,\A) -- (\W,\K);\draw[\styleGrille] (\Y,\A) -- (\Y,\K);

             \draw[\styleGrille] (\ZZ,\A) -- (\ZZ,\K);
             \draw[\styleGrille] (\Za,\A) -- (\Za,\K);
             \draw[\styleGrille] (\Zb,\A) -- (\Zb,\K);
             \draw[\styleGrille] (\Zc,\A) -- (\Zc,\K);

            \draw[solid,line width=0.4mm,color=red] (\A,\A)--(\E,\E) -- (\G,\C) -- (\M,\I) -- (\Q,\E) -- (\U,\I) -- (\ZZ,\C) -- (\Za,\E) -- (\Zc,\A);
         \end{tikzpicture}
		\end{center}
	\caption{Illustration of the bijection $\Phi$ between multisets and lattice paths}
\label{fig1}
\end{figure}

   Now, let us define the set $\mathcal{M}^\star_{n,k}$  of  $n$-multisets of $[k]$ with no consecutive integers, {\it i.e.}, multisets $\pi$ such
   that $\pi_{i+1}\neq\pi_i+1$ for all $i\in [n-1]$. Then we set  $\mathcal{M}^\star_n=\mathcal{M}^\star_{n,n}$,
   $\mathcal{M}^{s,\star}_{n}=\mathcal{M}^\star_{n}\cap\mathcal{M}^s_{n}$, $\mathcal{M}^\star=\bigcup_{n\geq 1} \mathcal{M}^\star_n$ and
    $\mathcal{M}^{s,\star}=\bigcup_{n\geq 1} \mathcal{M}^{s,\star}_n$. On the other hand, if $\mathcal{P}$ is a set of lattice paths consisting of $U$ and $D$ steps, then we denote by $\mathcal{P}^\star$ the subset of $\mathcal{P}$ consisting of paths that do not contain any occurrence of the pattern $DUD$.

   Considering these notations, it is straightforward to obtain the following theorem.

\begin{thm}\label{thmMGDstar}
The map $\Phi$ induces a bijection from $\mathcal{M}^\star_n$ to $\mathcal{GD}^\star_n$, and from  $\mathcal{M}^{s,\star}_n$ to $\mathcal{D}^\star_n$.
\end{thm}


It is well known (see for instance \cite{Mer,Sap,Sun}) that the cardinality of $\mathcal{D}^\star_n$ is  given by the general term of Motzkin sequence A001006 in \cite{Slo}.
 Then, using Theorem~\ref{thmMGDstar} the cardinality of the set $\mathcal{M}^{s,\star}_n$ is also counted by this sequence.

 Now, using combinatorial arguments, we prove that the cardinality of $\mathcal{M}^\star_n$
 is given by the general term of the  sequence A005773 in \cite{Slo} which also counts  directed animals with a given area on the square lattice.

\begin{thm}\label{thmMstarGenFun}
  The ordinary generating function for
  $n$-multisets of $[n]$ with no consecutive integers
  is
  \[
  {\frac {1-3\,z-\sqrt {1-2z-3\,{z}^{2}}}{6\,z-2}}.
  \]
\end{thm}

\begin{table}[t]
\begin{center}
\scalebox{1}{$\begin{array}{c|ccccccccc}
k\backslash n & 1 & 2 & 3 & 4 & 5 & 6 & 7 & 8&9\\
\hline
1 & 1 & 1& 1& 1& 1& 1 & 1 & 1 &1\\
2 & 2 & 2 & 2& 2& 2& 2 & 2 & 2 &2\\
3 & 3  &4  & 5&  6& 7 &  8 & 9 &10&11\\
4 & 4  &7  &10  &13&   16& 19  &  22&25&28   \\
5 & 5 &11 &18  &26  &35   &  45&  56 &  68 &81  \\
6 & 6 &16  &30  &48  &70   &96    & 126  & 160  &198 \\
\end{array}$
}
\end{center}\caption
{The number of $n$-multisets of $[k]$ with no consecutive integers}
\label{tab1}\end{table}

\begin{proof}
Let $f(z,u)=\sum_{n,k\geq 1}f_{n,k}\,z^nu^k$ be the bivariate generating function for the set $\mathcal{M}^\star_{n,k}$.
That is, the coefficient $f_{n,k}$ is the number of $n$-multisets of $[k]$ with no consecutive integers.
We build such a multiset by considering each integer from $\{1,\ldots ,k\}$ in turn, and marking how many times it occurs in the multiset.
Each of the $k$ integers contributes a factor of $u$, and each occurrence in the multiset contributes a factor of $z$.

Using this approach, $f(z,u)$ can be written in the following form:
\begin{equation}\label{eqfzu}
f(z,u)  \;=\;  \Seq[u] \:\times\:  u\,\Seqplus[z] \:\times\: \Seq\big[\Seqplus[u]\,u\,\Seqplus[z]\big] \:\times\:  \Seq[u]
,
\end{equation}
in which we make use of notation of Flajolet and Sedgewick~\cite{FS}: $\Seq[x]=1/(1-x)=1+x+x^2+\dots$ represents the occurrence of zero or more items counted by $x$, and $\Seqplus[x]=x/(1-x)=x+x^2+\dots$ represents the occurrence of one or more items.

The first term of~\eqref{eqfzu}, $\Seq[u]$, represents the (possibly empty) initial sequence of integers not in the multiset.
The second term, $u\Seqplus[z]$, represents the first integer in the multiset, occurring one or more times.

Each subsequent integer, if any, that occurs one or more times in the multiset is preceded by at least one integer not in the multiset, since it does not contain consecutive integers. So each such additional integer in the multiset is represented by $\Seqplus[u]u\Seqplus[z]$. The third term of~\eqref{eqfzu} thus represents all of the subsequent integers in the multiset. Finally, the fourth term, $\Seq[u]$, represents the (possibly empty) final sequence of integers not in the multiset.

Expansion and simplification yields
\[
f(z,u) \;=\; {\frac {uz}{ \left( 1-u \right)\left( 1-u-z+uz-{u}^{2}z \right)  }}
.
\]
Small values of $f_{n,k}$ are shown in Table~\ref{tab1}.

As a consequence of the bijection $\Phi$, the set of lattice paths of length $n+k$ starting at $(0,0)$, ending at $(n+k,n-k)$ consisting of $n$ up steps
and $k$ down steps, starting with an up-step and avoiding the pattern $DUD$ has a bivariate generating function given by $f(zu,z/u)$.

In order to obtain the generating function for the set $\mathcal{M}^{\star}_{n}=\mathcal{M}^\star_{n,n}$, we require the diagonal $\Delta(f)(z)=\sum_{n\geq 1}f_{n,n}z^n= [u^0]f(z/u, u)$, where $[u^0]g(u)$ is the constant coefficient of $u$ in $g(u)$.

If $g(u)=g(u,z)$ is a formal Laurent series, then the constant term $[u^0]g(u)$ is given by the sum of the residues of $u^{-1}g(u)$ at
those poles $\alpha$ of $g(u)$ for which $\lim_{z\to0}\alpha(z)=0$ (see~\cite[Section 6.3]{Stan2}).

In our case, $f(z/u, u)$ has a single pole $\alpha(z)$ for which $\alpha(0)=0$, and the residue of $u^{-1}f(z/u, u)$ at $\alpha(z)$ is
\[
{\frac {1-3\,z-\sqrt {1-2z-3\,{z}^{2}}}{6\,z-2}}
\]
as required.
\end{proof}

As a consequence of the bijection $\Phi$, we have the following.
\begin{cor}\label{corAvDUD}
   The ordinary generating function for Grand-Dyck paths of semilength $n$ starting with an up-step and avoiding the pattern $DUD$ is counted by sequence A005773 in \cite{Slo}.
\end{cor}

Sequence A005773 in \cite{Slo} counts diverse combinatorial objects, including various other types of lattice paths.
For example, Banderier et al prove in~\cite{Band} that Motzkin meanders (prefixes of Motzkin paths) are enumerated by this sequence. Sapanoukis et al.~\cite{Sap} state that Grand-Dyck paths of semilength $n$ starting with an up-step and avoiding the pattern $UDU$ are also counted by this sequence. We are unaware of a published proof of this, so present one very briefly here.
We use the fact that the construction behind $\Phi$ induces a bijection between these Grand-Dyck paths avoiding $UDU$
and $n$-multisets of $[n]$ in which no integer except $n$ occurs exactly once.

\begin{thm}\label{thmAvUDUMultisets}
  The set of $n$-multisets of $[n]$ in which no integer except $n$ occurs exactly once is counted by sequence A005773 in \cite{Slo}.
\end{thm}
\begin{proof}
  Let $h(z,u)$ be the bivariate generating function for $n$-multisets of $[k]$ in which no integer except $k$ occurs exactly once.
  We have
  \[
  h(z,u)
  \;=\; \Seq\big[ u\,\big( \Seq[z]-z \big) \big] \:\times\: u\,\Seq[z]
  \;=\; \frac {u}{1-z-u\,(1-z+z^2)}
  ,
  \]
  where the first term of the construction represents no occurrence or at least two occurrences of each integer from $\{1,\dots,k-1\}$, and the second term represents zero or more occurrences of $k$. Extracting the diagonal then yields
  \[
  [u^0]h(z/u, u) \;=\; {\frac {1-3\,z-\sqrt {1-2z-3\,{z}^{2}}}{6\,z-2}}
  \]
  as required.
\end{proof}

\begin{cor}\label{corAvUDU}
   The ordinary generating function for Grand-Dyck paths of semilength $n$ starting with an up-step and avoiding the pattern $UDU$ is counted by sequence A005773 in \cite{Slo}.
\end{cor}

\section{From multisets to directed animals via Grand-Dyck paths}

A {\it directed animal} $A$ of area $n$ (or equivalently with $n$ nodes) in the triangular lattice is a subset of $n$ points  in the lattice
 containing $(0,0)$ and where any point in $A$ can be reached from $(0,0)$ with up-moves $(0,1)$, right-moves $(1,0)$ and diagonal moves $(1,1)$ by staying always in $A$. Directed animals in the square lattice are those that do not use diagonal moves. See the left part of Figure~\ref{fig2} for an example of directed animal in the triangular lattice, and we refer to references in Introduction for several combinatorial studies on these objects. Let $\mathcal{Q}_n$ (resp. $\mathcal{T}_n$) be the set of directed animals with $n$ nodes in the square (resp. triangular) lattice, then its cardinality is given by the $n$th term of the sequence  A005773 in \cite{Slo} (resp. by the binomial coefficient $\binom{2n-1}{n}$).
We set $\mathcal{Q}=\cup_{n\geq 1}\mathcal{Q}_n$, $\mathcal{T}=\cup_{n\geq 1}\mathcal{T}_n$ and obviously we have $\mathcal{Q}\subset\mathcal{T}$.

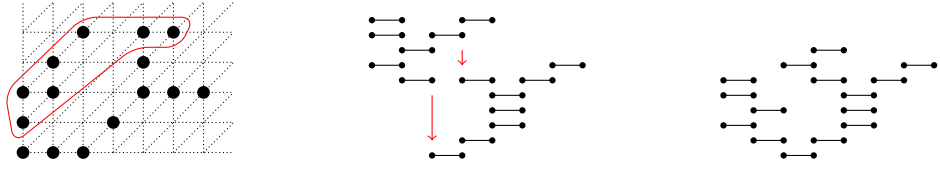
\begin{figure}[ht]
	
		\begin{center}
			\begin{tikzpicture}[scale=0.4]
				\draw [\styleGrille] (\A,\A) -- (\A,\F);
				\draw [\styleGrille] (\B,\A) -- (\B,\F);
				\draw [\styleGrille] (\C,\A) -- (\C,\F);
				\draw [\styleGrille] (\D,\A) -- (\D,\F);
				\draw [\styleGrille] (\E,\A) -- (\E,\F);
				\draw [\styleGrille] (\F,\A) -- (\F,\F);
				\draw [\styleGrille] (\G,\A) -- (\G,\F);
				\draw [\styleGrille] (\A,\A) -- (\H,\A);
				\draw [\styleGrille] (\A,\B) -- (\H,\B);
				\draw [\styleGrille] (\A,\C) -- (\H,\C);
				\draw [\styleGrille] (\A,\D) -- (\H,\D);
				\draw [\styleGrille] (\A,\E) -- (\H,\E);

				\draw [\styleGrille] (\A,\A) -- (\F,\F);
				\draw [\styleGrille] (\A,\B) -- (\E,\F);
				\draw [\styleGrille] (\A,\C) -- (\D,\F);
				\draw [\styleGrille] (\A,\D) -- (\C,\F);
				\draw [\styleGrille] (\A,\E) -- (\B,\F);
				
                \draw [\styleGrille] (\B,\A) -- (\G,\F);
				\draw [\styleGrille] (\C,\A) -- (\H,\F);
				\draw [\styleGrille] (\D,\A) -- (\H,\E);
				\draw [\styleGrille] (\E,\A) -- (\H,\D);
				\draw [\styleGrille] (\F,\A) -- (\H,\C);
				\draw [\styleGrille] (\G,\A) -- (\H,\B);


                \pointt{\A}{\A} 
				\pointt{\A}{\B} 
				\pointt{\B}{\A} 
				\pointt{\A}{\C} 
				\pointt{\B}{\C} 

				\pointt{\C}{\A} 
				\pointt{\D}{\B} 
				\pointt{\E}{\C} 

                \pointt{\G}{\C} 
				\pointt{\E}{\E} 
				\pointt{\E}{\D} 
				\pointt{\F}{\E} 

                 \pointt{\B}{\D} 
				\pointt{\C}{\E} 
				\pointt{\F}{\C} 
\draw[xshift=17mm,rounded corners=5pt, red](-1.5,0.8)--(-1.8,\C)--(0.8,\f)--(\E,\f)--(\e,\e)--(\C,\e)--cycle;
			\end{tikzpicture}
\qquad \qquad
          \begin{tikzpicture}[scale=0.4  ]

				\draw (\a,\a) -- (\b,\a);
                \point{\a}{\a} 
				\point{\b}{\a} 
                 \draw (\b,\A) -- (\c,\A);
                \point{\b}{\A} 
				\point{\c}{\A} 

				\draw (\c,\b) -- (\d,\b);
                \point{\d}{\b} 
				\point{\c}{\b} 
                 \draw (\c,\B) -- (\d,\B);
                \point{\d}{\B} 
				\point{\c}{\B} 

				\draw (\c,\c) -- (\d,\c);
                \point{\c}{\c} 
				\point{\d}{\c} 
                 \draw (\d,\C) -- (\e,\C);
                \point{\d}{\C} 
				\point{\e}{\C} 

                 \draw (\e,\d) -- (\f,\d);
                \point{\e}{\d} 
				\point{\f}{\d} 

				\draw (\c,\C) -- (\b,\C);
                \point{\b}{\C} 
				\point{\c}{\C} 

                 \draw (\b,\e) -- (\a,\e);
                \point{\b}{\e} 
				\point{\a}{\e} 
				\draw (\b,\E) -- (\c,\E);
                \point{\b}{\E} 
				\point{\c}{\E} 
               \draw (\a,\C) -- (-1,\C);
                \point{\a}{\C} 
				\point{-1cm}{\C} 
                 \draw (-1,\d) -- (-2,\d);
                \point{-2cm}{\d} 
				\point{-1cm}{\d} %
               \draw (-1,\D) -- (\a,\D);
                \point{\a}{\D} 
				\point{-1cm}{\D} 
                 \draw (-1,\e) -- (-2,\e);
                \point{-2cm}{\e} 
				\point{-1cm}{\e} 
                 \draw (-2,\E) -- (-1,\E);
                \point{-1cm}{\E} 
				\point{-2cm}{\E} 
               \draw [->,red] (0,\c) -- (0,\A);
               \draw [->,red] (\b,\D) -- (\b,\d);
			\end{tikzpicture}
\qquad \qquad
          \begin{tikzpicture}[scale=0.4  ]

				\draw (\a,\a) -- (\b,\a);
                \point{\a}{\a} 
				\point{\b}{\a} 
                 \draw (\b,\A) -- (\c,\A);
                \point{\b}{\A} 
				\point{\c}{\A} 

				\draw (\c,\b) -- (\d,\b);
                \point{\d}{\b} 
				\point{\c}{\b} 
                 \draw (\c,\B) -- (\d,\B);
                \point{\c}{\B} 
				\point{\d}{\B} 

				\draw (\c,\c) -- (\d,\c);
                \point{\c}{\c} 
				\point{\d}{\c} 
                 \draw (\d,\C) -- (\e,\C);
                \point{\d}{\C} 
				\point{\e}{\C} 

                 \draw (\e,\d) -- (\f,\d);
                \point{\e}{\d} 
				\point{\f}{\d} 

				\draw (\c,\C) -- (\b,\C);
                \point{\b}{\C} 
				\point{\c}{\C} 
                 \draw (\b,\d) -- (\a,\d);
                \point{\b}{\d} 
				\point{\a}{\d} 

				\draw (\b,\D) -- (\c,\D);
                \point{\b}{\D} 
				\point{\c}{\D} 

               \draw (\a,\A) -- (-1,\A);
                \point{\a}{\A} 
				\point{-1cm}{\A} 
                 \draw (-1,\b) -- (-2,\b);
                \point{-2cm}{\b} 
				\point{-1cm}{\b} 

               \draw (-1,\B) -- (\a,\B);
                \point{\a}{\B} 
				\point{-1cm}{\B} 
                 \draw (-1,\c) -- (-2,\c);
                \point{-2cm}{\c} 
				\point{-1cm}{\c} 

                 \draw (-2,\C) -- (-1,\C);
                \point{-1cm}{\C} 
				\point{-2cm}{\C} 
			\end{tikzpicture}
\end{center}
	\caption{A directed animal and its associated heap.}
\label{fig2}
\end{figure}

In the literature \cite{Bou,Vie,Vie2}, directed animals are often viewed as heaps obtained by dropping vertically dimers such that each dimer (except the first) touches
the one below by at least one  of its extremities. Indeed, from $A\in\mathcal{T}$, we  apply  a counterclockwise rotation of 45 degree of its graphical representation and
we replace  each point of $A$ with a dimer of width $\sqrt{2}/2$. See Figure~\ref{fig2} for an example of such a representation. Notice that directed animals
in $\mathcal{Q}$ correspond to heaps of dimers where no dimer has another dimer directly above it (such a heap will be
 called {\it strict}). Let $\mathcal{T}^s$ (resp. $\mathcal{Q}^s$) be the set of all subdiagonal directed animals in $\mathcal{T}$ (resp. $\mathcal{Q}$),
 {\it i.e.}, directed animals where all its points $(i,j)$ satisfy $j\leq i$.

Without losing accuracy, the sets  $\mathcal{T}$ and $\mathcal{Q}$ will also be used to designate respectively the set of heaps of dimers and the
set of strict heaps of dimers. Then, any  heap $A\in \mathcal{T}^s$ has a unique factorization into one of the four following forms (see  \cite{Bou}):

\begin{center} ($i$) $\begin{tikzpicture}[scale=0.25]
				
				\draw (\a,\a) -- (\b,\a);
                \point{\a}{\a} 
				\point{\b}{\a} 
			\end{tikzpicture}$~~~~~ ($ii$) $\begin{tikzpicture}[scale=0.25,baseline=0.5em]
				
				\draw (\a,\a) -- (\b,\a);
                \point{\a}{\a} 
				\point{\b}{\a} 
                \draw (\b,\A) -- (\c,\A);
                \point{\b}{\A} 
				\point{\c}{\A} 
                \draw [solid, line width=0.2mm] (\b,\A) -- (\b,\D);
                \draw [solid, line width=0.2mm] (\c,\A) -- (\f,\D);
                \draw (\C,\C) node {$B$};
			\end{tikzpicture}$ ~~~~~($iii$) $\begin{tikzpicture}[scale=0.25,baseline=0.5em]
				
				\draw (\a,\a) -- (\b,\a);
                \point{\a}{\a} 
				\point{\b}{\a} 
                \draw (\a,\A) -- (\b,\A);
                \point{\a}{\A} 
				\point{\b}{\A} 
                \draw [solid, line width=0.2mm] (\a,\A) -- (\a,\D);
                \draw [solid, line width=0.2mm] (\b,\A) -- (\e,\D);
                \draw (\B,\C) node {$B$};
			\end{tikzpicture}$ ~~~~~ ($iv$) $\begin{tikzpicture}[scale=0.25,baseline=0.5em]
				
				\draw (\a,\a) -- (\b,\a);
                \point{\a}{\a} 
				\point{\b}{\a} 
                \draw (\b,\A) -- (\c,\A);
                \point{\b}{\A} 
				\point{\c}{\A} 
                \draw [solid, line width=0.2mm] (\b,\A) -- (\b,\B);
                \draw [solid, line width=0.2mm] (\c,\A) -- (\f,\D);
                \draw (\a,\c) -- (\b,\c);
                \point{\a}{\c} 
				\point{\b}{\c} 
                \draw [solid, line width=0.2mm] (\a,\c) -- (\a,\f);
                \draw [solid, line width=0.2mm] (\b,\c) -- (\e,\f);
                \draw (\b,\D) node {$C$};\draw (\d,\C) node {$B$};
			\end{tikzpicture}$,
\end{center}
\vskip0.2cm where $B,C\in \mathcal{T}^s$.
Moreover, any heap $A\in \mathcal{T}\backslash\mathcal{T}^s$ has a unique factorization:

 \begin{center}
 ($v$) \begin{tikzpicture}[scale=0.25,baseline=0.5em]
				
                \draw (\b,\a) -- (\c,\a);
                \point{\b}{\a} 
				\point{\c}{\a} 
                \draw [solid, line width=0.2mm] (\b,\a) -- (\b,\b);
                \draw [solid, line width=0.2mm] (\c,\a) -- (\f,\d);
                \draw (\a,\B) -- (\b,\B);
                \point{\a}{\B} 
				\point{\b}{\B} 
                \draw [solid, line width=0.2mm] (\b,\B) -- (\e,\E);
                \draw [solid, line width=0.2mm] (\a,\B) -- (-1,\E);
                \draw [solid, line width=0.2mm] (\b,\B) -- (\e,\E);
                 \draw (\d,\c) node {$B$};
                 \draw (\A,\D) node {$C$};
			\end{tikzpicture},
\end{center}
     where $B\in\mathcal{T}^s$ and $C\in\mathcal{T}$.

The factorization of $A\in\mathcal{Q}$ (resp. $A\in\mathcal{Q}^s$) is obtained after omitting the case ($iii$). Translating these factorizations using functional equations involving the generating functions $T(z)$ and $T^s(z)$ for $\mathcal{T}$ and $\mathcal{T}^s$ (resp. $Q(z)$ and $Q^s(z)$ for $\mathcal{Q}$ and $\mathcal{Q}^s$), we obtain

\[
T^s(z)={\frac {1-2z-\sqrt {1-4z}}{2z}}, \quad T(z)={\frac {1-4\,z-\sqrt {1-4z}}{8\,z-2}},
\]

\[
Q^s(z)={\frac {1-z-\sqrt {1-2z-3\,{z}^{2}}}{2z}}, \mbox{ and } Q(z)={\frac {1-3\,z-\sqrt {1-2z-3\,{z}^{2}}}{6\,z-2}}.
\]

The coefficients of $z^n$ in the Taylor expansion of $Q^s(z)$ (resp. $Q(z)$, $T^s(z)$ and $T(z)$) generate a shift of the Motzkin
sequence A001006  (resp. A005773, the Catalan sequence A000108 and A001700) in \cite{Slo}).

Now, we construct a bijection from $\mathcal{M}_{n}$ to the set $\mathcal{T}_n$  of directed animals in the triangular lattice which
transports $\mathcal{M}^{\star}_{n}$ into $\mathcal{Q}_n$. We proceed in two steps. Firstly, we define a bijection  from $\mathcal{T}^s_n$ to $\mathcal{M}^s_n$ for $n\geq 1$, and secondly we extend it from $\mathcal{T}_n$ to
$\mathcal{M}_n$. For the first step, and according to the above bijection $\Phi$ from $\mathcal{M}^s_n$ to $\mathcal{D}_n$, it suffices to define a one-to-one correspondence $\Psi$ between $\mathcal{T}^s_n$ and
 $\mathcal{D}_n$. Letting $A$ be a directed animal in  $\mathcal{T}^s$, we define $\Psi(A)$ with respect to its four possible factorizations:
\begin{itemize}[itemsep=3pt]
\item if $A$ satisfies ($i$) then  $\Psi(A)=UD$,

\item if $A$ satisfies ($ii$) then $\Psi(A)=U \Psi(B) D$,

\item if $A$ satisfies ($iii$) then $\Psi(A)=\Psi(B) UD$,

\item if $A$ satisfies ($iv$) then   $\Psi(A)=\Psi(C)U \Psi(B) D$.
\end{itemize}
Due to the recursive definition, the image by $\Psi$ of a directed animal in  $\mathcal{T}^s_n$  is a Dyck path of semilength $n$, and the image of an
 element of $\mathcal{Q}^s_n$ is a Dyck path with no pattern $DUD$, {\it i.e.}  in $\mathcal{D}^\star_n$.

\begin{thm}\label{thmTsMs}
For $n\geq 1$, the map $\Phi^{-1}\cdot \Psi$ is a bijection from $\mathcal{T}^s_n$ to  $\mathcal{M}^{s}_n$, and the image of $\mathcal{Q}^{s}_n$ is $\mathcal{M}^{s,\star}_n$.
\end{thm}
\begin{proof}
Since $\Phi^{-1}$ is a bijection from $\mathcal{D}_n$ to  $\mathcal{M}^{s}_n$, it suffices to prove that $\Psi$ is
 a bijection from $\mathcal{T}^s_n$ to $\mathcal{D}_n$. As these two last sets are both enumerated by the Catalan numbers,
  it suffices to prove the injectivity of $\Psi$.
  We proceed by induction on $n$. The case $n=1$ holds trivially. We assume that $\Psi$ is injective for $k\leq n$, and we prove the result for $n+1$.
  By definition, the image by $\Psi$ of animals satisfying ($i$) and $(ii)$ are Dyck paths with only one return on the $x$-axis, {\it i.e.} with only one down step $D$ that touches the $x$-axis.
   Animals satisfying ($iii$) are sent by $\Psi$
   to Dyck paths ending with $DUD$ and with at least two return on the $x$-axis. Animals satisfying ($iv$) are sent to Dyck paths with at least two down steps at the end, and with at least two returns. Then, for $A,A'\in \mathcal{T}^s_{n+1}$,
  $\Psi(A)=\Psi(A')$ implies that $A$ and $A'$ belong to the same case ($i$), ($ii$), ($iii$) or ($iv$). The recurrence hypothesis induces $A=A'$ which completes
  the induction.
  Moreover, in the case where $A\in \mathcal{Q}^s_n$, it does not satisfy $(iii)$ and this  implies that $\Psi(A)$ is a Dyck path avoiding $DUD$. Finally, a
  cardinality argument proves that $\Psi(\mathcal{Q}^{s}_n)=\mathcal{M}^{s,\star}_n$.
\end{proof}

Now we  extend the map $\Psi$ from $\mathcal{T}_n$ to $\mathcal{M}_n$ as follows. Let $A$ be a directed animal in  $\mathcal{T}_n\backslash \mathcal{T}^s_n$,
then $A$ can be factorized as ($v$) with $B\in\mathcal{T}^s$ and $C\in\mathcal{T}$. In the subcase where $C\in \mathcal{T}\backslash \mathcal{T}^s$, $C$ satisfies
the case ($v$), and let $D\in\mathcal{T}^s$, $E\in\mathcal{T}$ be the two parts of its factorization. According to these two cases, we set:

\[
\Psi(A)=\left\{\begin{array}{ll}
        \Psi(B)\Psi(C)^r &  \mbox{ if } C\in \mathcal{T}^s,\\
        \Psi(B)\Psi(D)^r\Psi(E) &  \mbox{ otherwise, }\\
   \end{array}\right.
\]
where $P^r$ is obtained from $P$ by reading the Dyck path $P$ from right to left (for instance, if $P=UUDUUDDD$ then $P^r=DDDUUDUU$). Less formally, $\Psi$ maps successive components from $\mathcal{T}^s$ to Dyck paths alternately above and below the $x$-axis. See Figure~\ref{fig3} for an
  illustration of the map $\Psi$.

\begin{thm}\label{thmTM}
For $n\geq 1$, the map $\Phi^{-1}\cdot \Psi$ is a bijection from $\mathcal{T}_n$ to  $\mathcal{M}_n$, and the image of $\mathcal{Q}_n$ is   $\mathcal{M}^{\star}_n$.
\end{thm}
\begin{proof}
Let us  prove that $\Psi$ is a bijection from $\mathcal{T}_n$ to $\mathcal{M}_n$. As these two sets have the same cardinality, it suffices to  prove the injectivity of $\Psi$.
Using Theorem~\ref{thmTsMs}, it remains to prove that directed animals $A\in\mathcal{T}_n\backslash\mathcal{T}_n^s$ are sent bijectively by $\Psi$ to Grand-Dyck paths in $\mathcal{GD}_n\backslash\mathcal{D}_n$.
 Due to the definition of $\Psi$ whenever $A\in\mathcal{T}_n\backslash\mathcal{T}_n^s$, we have either $\Psi(A)=\Psi(B)\Psi(C)^r$ or $\Psi(A)=\Psi(B)\Psi(D)^r\Psi(E)$ with $B,C,D\in\mathcal{T}^s$ and $E\in\mathcal{T}$.
 Then, the path $\Psi(A)$ starts with an up-step (the first step of the non-empty  Dyck path $\Psi(B)$), and since the first step of $\Psi(C)^r$ (resp. $\Psi(D)^r$) is a down-step, $\Psi(A)$ crosses the $x$-axis which ensures that $\Psi(B)\Psi(C)^r$ (resp. $\Psi(B)\Psi(D)^r$) belongs to $\mathcal{GD}_n\backslash\mathcal{D}_n$. We complete the proof with a simple induction on $n$.
Whenever $A\in\mathcal{Q}$, Theorem~\ref{thmTsMs} ensures that $\Psi(B)$, $\Psi(C)$ and $\Psi(D)$ avoid the pattern $DUD$. By symmetry, the paths $\Psi(C)^r$ and $\Psi(D)$ avoid $DUD$ which implies that the Grand-Dyck
 paths $\Psi(B)\Psi(C)^r$ and $\Psi(B)\Psi(D)^r$ do not contain $DUD$. By induction, $\Psi(A)$ belongs to $\mathcal{M}^{\star}_n$.
\end{proof}

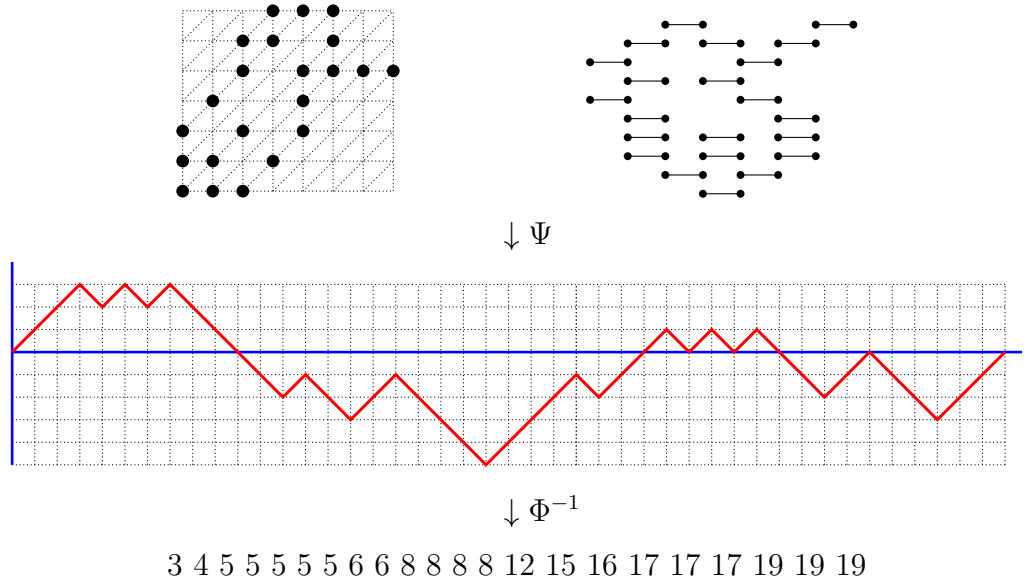
\begin{figure}[ht]
	
		\begin{center}
\begin{tikzpicture}[scale=0.4]
				\draw [\styleGrille] (\A,\A) -- (\A,\G);
				\draw [\styleGrille] (\B,\A) -- (\B,\G);
				\draw [\styleGrille] (\C,\A) -- (\C,\G);
				\draw [\styleGrille] (\D,\A) -- (\D,\G);
				\draw [\styleGrille] (\E,\A) -- (\E,\G);
				\draw [\styleGrille] (\F,\A) -- (\F,\G);
				\draw [\styleGrille] (\G,\A) -- (\G,\G);
				\draw [\styleGrille] (\H,\A) -- (\H,\G);
				
				\draw [\styleGrille] (\A,\A) -- (\H,\A);
				\draw [\styleGrille] (\A,\B) -- (\H,\B);
				\draw [\styleGrille] (\A,\C) -- (\H,\C);
				\draw [\styleGrille] (\A,\D) -- (\H,\D);
				\draw [\styleGrille] (\A,\E) -- (\H,\E);
				\draw [\styleGrille] (\A,\F) -- (\H,\F);
				\draw [\styleGrille] (\A,\G) -- (\H,\G);

				\draw [\styleGrille] (\A,\A) -- (\G,\G);
				\draw [\styleGrille] (\A,\B) -- (\F,\G);
				\draw [\styleGrille] (\A,\C) -- (\E,\G);
				\draw [\styleGrille] (\A,\D) -- (\D,\G);
				\draw [\styleGrille] (\A,\E) -- (\C,\G);
				\draw [\styleGrille] (\A,\F) -- (\B,\G);

                \draw [\styleGrille] (\B,\A) -- (\H,\G);
				\draw [\styleGrille] (\C,\A) -- (\H,\F);
				\draw [\styleGrille] (\D,\A) -- (\H,\E);
				\draw [\styleGrille] (\E,\A) -- (\H,\D);
				\draw [\styleGrille] (\F,\A) -- (\H,\C);
				\draw [\styleGrille] (\G,\A) -- (\H,\B);


                \pointt{\A}{\A} 
				\pointt{\A}{\B} 
				\pointt{\B}{\A} 
				\pointt{\A}{\C} 
				
                \pointt{\B}{\B} 
                \pointt{\C}{\C} 

				 \pointt{\C}{\F} 
 \pointt{\D}{\F} 
 \pointt{\D}{\G} 
  \pointt{\E}{\G} 
   \pointt{\F}{\G} 
                   \pointt{\C}{\A} 
				\pointt{\D}{\B} 
				\pointt{\E}{\C} 

                \pointt{\G}{\E} 
				\pointt{\E}{\E} 
				\pointt{\E}{\D} 
				\pointt{\F}{\E} 

                 \pointt{\B}{\D} 
				\pointt{\C}{\E} 
				\pointt{\H}{\E} 
                \pointt{\F}{\F}

			\end{tikzpicture}\qquad\qquad\qquad
\begin{tikzpicture}[scale=0.5  ]

				\draw (\a,\a) -- (\b,\a);
                \point{\a}{\a} 
				\point{\b}{\a} 

                \draw (\b,\A) -- (\c,\A);
                \point{\b}{\A} 
				\point{\c}{\A} 

				\draw (\c,\b) -- (\d,\b);
                \point{\d}{\b} 
				\point{\c}{\b} 

                 \draw (\c,\B) -- (\d,\B);
                \point{\c}{\B} 
				\point{\d}{\B} 

				\draw (\c,\c) -- (\d,\c);
                \point{\c}{\c} 
				\point{\d}{\c} 

                 \draw (\c,\e) -- (\d,\e);
                \point{\d}{\e} 
				\point{\c}{\e} 

                 \draw (\d,\E) -- (\e,\E);
                \point{\d}{\E} 
				\point{\e}{\E} 

				\draw (\c,\C) -- (\b,\C);
                \point{\b}{\C} 
				\point{\c}{\C} 

                 \draw (\b,\d) -- (\a,\d);
                \point{\b}{\d} 
				\point{\a}{\d} 

				\draw (\b,\D) -- (\c,\D);
                \point{\b}{\D} 
				\point{\c}{\D} 

               \draw (\a,\A) -- (-1,\A);
                \point{\a}{\A} 
				\point{-1cm}{\A} 

                 \draw (-1,\b) -- (-2,\b);
                \point{-2cm}{\b} 
				\point{-1cm}{\b} 

               \draw (-1,\B) -- (-2,\B);
                \point{-2cm}{\B} 
				\point{-1cm}{\B} 

                 \draw (-1,\c) -- (-2,\c);
                \point{-2cm}{\c} 
				\point{-1cm}{\c} 

                 \draw (-2,\C) -- (-3,\C);
                \point{-3cm}{\C} 
				\point{-2cm}{\C} 

                \draw (\a,\b) -- (\b,\b);
                \point{\a}{\b} 
				\point{\b}{\b} 

                \draw (\a,\B) -- (\b,\B);
                \point{\a}{\B} 
				\point{\b}{\B} 

                \draw (-1,\d) -- (-2,\d);
                \point{-2cm}{\d} 
				\point{-1cm}{\d} 

                 \draw (-2,\D) -- (-3,\D);
                \point{-3cm}{\D} 
				\point{-2cm}{\D} 

                \draw (-1,\e) -- (-2,\e);
                \point{-2cm}{\e} 
				\point{-1cm}{\e} 

                 \draw (\a,\e) -- (\b,\e);
                \point{\a}{\e} 
				\point{\b}{\e} 

                 \draw (-1cm,\E) -- (\a,\E);
                \point{\a}{\E} 
				\point{-1cm}{\E} 
			\end{tikzpicture}

\smallskip
$\phantom{\Psi}\downarrow\Psi$
\smallskip

\begin{tikzpicture}[scale=0.15]
            \draw[\styleGrille] (\A,\A)-- (88.5cm,\A);
             \draw[\styleGrille] (\A,\E)-- (88.5cm,\E);
              \draw[solid,line width=0.35mm,color=blue](\A,\E)-- (90cm,\E);
               \draw[\styleGrille] (\A,\G)-- (88.5cm,\G);
               \draw[\styleGrille] (\A,\I)-- (88.5cm,\I);
              \draw[\styleGrille] (\A,\K)-- (88.5cm,\K);
              \draw[\styleGrille] (\A,\C)-- (88.5cm,\C);
              \draw[\styleGrille] (\A,-1.5cm)-- (88.5cm,-1.5cm);
              \draw[\styleGrille] (\A,-3.5cm)-- (88.5cm,-3.5cm);
              \draw[\styleGrille] (\A,-5.5cm)-- (88.5cm,-5.5cm);              
            \draw[solid,line width=0.35mm,color=blue] (\A,-5.5cm) -- (\A,\M);
             \draw[\styleGrille] (\C,-5.5cm) -- (\C,\K);\draw[\styleGrille] (\E,-5.5cm) -- (\E,\K);\draw[\styleGrille] (\G,-5.5cm) -- (\G,\K);
             \draw[\styleGrille] (\I,-5.5cm) -- (\I,\K);\draw[\styleGrille] (\K,-5.5cm) -- (\K,\K);\draw[\styleGrille] (\M,-5.5cm) -- (\M,\K);
             \draw[\styleGrille] (\O,-5.5cm) -- (\O,\K);\draw[\styleGrille] (\Q,-5.5cm) -- (\Q,\K);\draw[\styleGrille] (\S,-5.5cm) -- (\S,\K);
             \draw[\styleGrille] (\U,-5.5cm) -- (\U,\K);\draw[\styleGrille] (\W,-5.5cm) -- (\W,\K);\draw[\styleGrille] (\Y,-5.5cm) -- (\Y,\K);

             \draw[\styleGrille] (\ZZ,-5.5cm) -- (\ZZ,\K);
             \draw[\styleGrille] (\Za,-5.5cm) -- (\Za,\K);
             \draw[\styleGrille] (\Zb,-5.5cm) -- (\Zb,\K);
             \draw[\styleGrille] (\Zc,-5.5cm) -- (\Zc,\K);
             \draw[\styleGrille] (\Zi,-5.5cm) -- (\Zi,\K);
             \draw[\styleGrille] (52.5cm,-5.5cm) -- (52.5cm,\K);
             \draw[\styleGrille] (54.5cm,-5.5cm) -- (54.5cm,\K);
             \draw[\styleGrille] (56.5cm,-5.5cm) -- (56.5cm,\K);
             \draw[\styleGrille] (58.5cm,-5.5cm) -- (58.5cm,\K);
             \draw[\styleGrille] (50.5cm,-5.5cm) -- (50.5cm,\K);
             \draw[\styleGrille] (48.5cm,-5.5cm) -- (48.5cm,\K);
             \draw[\styleGrille] (46.5cm,-5.5cm) -- (46.5cm,\K);
             \draw[\styleGrille] (44.5cm,-5.5cm) -- (44.5cm,\K);
             \draw[\styleGrille] (36.5cm,-5.5cm) -- (36.5cm,\K);
             \draw[\styleGrille] (38.5cm,-5.5cm) -- (38.5cm,\K);
             \draw[\styleGrille] (40.5cm,-5.5cm) -- (40.5cm,\K);
             \draw[\styleGrille] (42.5cm,-5.5cm) -- (42.5cm,\K);
             \draw[\styleGrille] (34.5cm,-5.5cm) -- (34.5cm,\K);
             \draw[\styleGrille] (60.5cm,-5.5cm) -- (60.5cm,\K);

             \draw[\styleGrille] (62.5cm,-5.5cm) -- (62.5cm,\K);
             \draw[\styleGrille] (64.5cm,-5.5cm) -- (64.5cm,\K);
             \draw[\styleGrille] (66.5cm,-5.5cm) -- (66.5cm,\K);
             \draw[\styleGrille] (68.5cm,-5.5cm) -- (68.5cm,\K);
             \draw[\styleGrille] (70.5cm,-5.5cm) -- (70.5cm,\K);
             \draw[\styleGrille] (72.5cm,-5.5cm) -- (72.5cm,\K);
             \draw[\styleGrille] (74.5cm,-5.5cm) -- (74.5cm,\K);
             \draw[\styleGrille] (76.5cm,-5.5cm) -- (76.5cm,\K);

             \draw[\styleGrille] (82.5cm,-5.5cm) -- (82.5cm,\K);
             \draw[\styleGrille] (84.5cm,-5.5cm) -- (84.5cm,\K);
             \draw[\styleGrille] (86.5cm,-5.5cm) -- (86.5cm,\K);
             \draw[\styleGrille] (88.5cm,-5.5cm) -- (88.5cm,\K);
             \draw[\styleGrille] (78.5cm,-5.5cm) -- (78.5cm,\K);
             \draw[\styleGrille] (80.5cm,-5.5cm) -- (80.5cm,\K);


            \draw[solid,line width=0.4mm,color=red] (\A,\E)--(\G,\K) -- (\I,\I) -- (\K,\K) -- (\M,\I) -- (\O,\K) -- (\Y,\A) -- (26.5cm,\C) -- (30.5cm,-1.5cm) -- (34.5cm,\C)--(42.5cm,-5.5cm)--(50.5cm,\C)--(52.5cm,\A)--(58.5cm,\G)--(60.5cm,\E)--(62.5cm,\G)--(64.5cm,\E)--(66.5cm,\G)--(72.5cm,\A)--(76.5cm,\E)--
            (82.5cm,-1.5cm)--(88.5cm,\E);
         \end{tikzpicture}

         \medskip
         $\phantom{\Phi^{-1}}\downarrow\Phi^{-1}$
         \medskip

          { $3~4~5~5~5~5~5~6~6~8~8~8~8~12~15~16~17~17~17~19~19~19$}
	\end{center}
	\caption{Bijection $\Phi^{-1}\Psi$ between directed animals  and multisets via Grand-Dyck paths.}
\label{fig3}
\end{figure}
Now we define some statistics and parameters on $\mathcal{T}_n$, $\mathcal{M}_n$ and $\mathcal{GD}_n$, and we show how the  bijections $\Phi$, $\Psi$ and $\Phi^{-1}\cdot \Psi$ establish correspondences between them.
Table~\ref{tab2} summarizes these correspondences.

For a directed animal $A\in\mathcal{T}_n$, we set:
\begin{itemize}[itemsep=3pt]
\item $\mathbf{Area}(A)=$ number of points in $A$,

\item $\mathbf{Lw}(A)=$ left width, {\it i.e.} $\max\{i\geq 0 \mbox{ such that  the line } y=x+i \mbox{ meets } A\}$,

\item $\mathbf{Rw}(A)=$ right width, {\it i.e.} $\max\{i\geq 1 \mbox{ such that  the line } y=x-i+1 \mbox{ meets } A\}$,

\item $\mathbf{Width}(A)=\mathbf{Lw}(A)+\mathbf{Rw}(A)=$ width,

\item $\mathbf{Diag}(A)=$ number of \raisebox{-0.22cm}{\DUD} in $A$, where $\times$ means a site without point in $A$,

\item $\mathbf{Nbp}(A,i)=$ number of points of $A$  on the line $y=x-i+1$,
\end{itemize}
\medskip
For a multiset $\pi\in\mathcal{M}_n$, we define $\delta(\pi_i)=0$ if $\pi_i<i$ and 1 otherwise, and we set:
\begin{itemize}[itemsep=3pt]
\item $\mathbf{Length}(\pi)=$ $n$,

\item $\mathbf{Cross}(\pi)= \mbox{card}\{i\in[n-1], \delta(\pi_i)\neq\delta(\pi_{i+1})\}$,

\item $\mathbf{Adj}(\pi)=$ number of adjacencies, {\it i.e.}, $\mbox{card}\{i\in[n-1], \mbox{such that } \pi_{i+1}=\pi_i+1\}$,

\item $\mathbf{Gap}{}(\pi,i)= |\pi_i-i|-c_i$ where $c_i= \mbox{card}\{j\leq i-1, \delta(\pi_j)\neq\delta(\pi_{j+1})\}$,

\item $\mathbf{Gap}(\pi)= \max_{i\in[n]}\mathbf{Gap}{}(\pi,i)$.
\end{itemize}
\medskip
For a Grand-Dyck path $P\in\mathcal{GD}_n$, the height {\bf h}$(a,b)$ of a point $(a,b)\in P$ is the ordinate $b$, and $\mathbf{h}(P)=\max\{\mathbf{h}(a,b): (a,b)\in P\}$. Here we consider a new height function
 defined by {\bf Height}$(a,b)=|b|-c_a$ where $c_a$ is the number of the $x$-axis crossings before the line $x=a$, and we set:
\begin{itemize}[itemsep=3pt]
\item $\mathbf{Semilength}(P)=$ number of up-steps $U$,

\item $\mathbf{Cross}(P)=$  number of crossings of the $x$-axis,

\item $\mathbf{Height}(P)= \max_{(a,b)\in P}\mathbf{Height}(a,b)$,

\item $\mathbf{Dud}(P)=$ number of pattern $DUD$,

\item $\mathbf{Nbu}(P,i)=$ number of $U$ having endpoint $(a,b)$ satisfying {\bf Height}$(a,b)=i+1$.

\end{itemize}

\begin{table}[ht]\begin{center}
\begin{tabular}{p{3.2cm} p{4.8cm} p{3.6cm}}
$A\in \mathcal{T}_n$ &$P=\Psi(A)\in\mathcal{GD}_{n}$& $\pi=\Phi^{-1}(P)\in\mathcal{M}_{n}$\\
\hline
$\mathbf{Area}(A)$ & $\mathbf{Semilength}(P)$ & $\mathbf{Length}(\pi)$\\

$\mathbf{Lw}(A)$ & $\mathbf{Cross}(P)$ & $\mathbf{Cross}(\pi)$ \\

$\mathbf{Rw}(A)$ & $\mathbf{Height}(P)$ & $\mathbf{Gap}(\pi)$ \\

$\mathbf{Width}(A)$ & $\mathbf{Cross}(P)+\mathbf{Height}(P)$ & $\mathbf{Cross}(\pi)+\mathbf{Gap}(\pi)$\\

$\mathbf{Diag}(A)$ & $\mathbf{Dud}(P)$ & $\mathbf{Adj}(\pi)$\\

$\mathbf{Nbp}(A,i)$ & $\mathbf{Nbu}(P,i)$ & $\mathbf{Gap}(\pi,i)$\\
\hline
\end{tabular}\end{center}
\caption{Statistic correspondences  by the bijections $\Psi$ and $\Phi$. }
\label{tab2}
\end{table}

\begin{thm}\label{thmStatistics}
The bijections $\Phi$ and $\Psi$ induce correspondences between statistics as summarized in Table~\ref{tab2}.
\end{thm}
\begin{proof}
The statistic correspondences induced by $\Phi$ are easy to check. So, we  only prove  the  correspondences generated by $\Psi$ from directed animals to Grand-Dyck paths.

When $A\in\mathcal{T}^s$, we have $\mathbf{Lw}(A)=\mathbf{Cross}(\Psi(A))=0$. When $A\in\mathcal{T}\backslash\mathcal{T}^s$, $A$ satisfies $(v)$ with
$B\in\mathcal{T}^s$ and $C\in\mathcal{T}$. Then $\mathbf{Lw}(A)=1+ \mathbf{Lw}(C)$. We assume the recurrence hypothesis $\mathbf{Lw}(C)=\mathbf{ Cross}(\Psi(C))$, which
implies $\mathbf{Lw}(A)=1+\mathbf{Cross}(\Psi(C))$. Using the recursive definition of $\Psi$, we have $\mathbf{Cross}(\Psi(A))=1+\mathbf{Cross}(\Psi(C))$ which
 gives by induction $\mathbf{Lw}(A)=\mathbf{Cross}(\Psi(A))$.

 When $A\in\mathcal{T}^s$, it satisfies  $(i)$, $(ii)$, $(iii)$ or $(iv)$, and the recursive definition of $\Psi$ implies that
 $\mathbf{Rw}(A)=\mathbf{h}(\Psi(A))=\mathbf{Height}(\Psi(A))$. Otherwise, if $A$  is factorized as $(v)$ with $B\in\mathcal{T}^s$ and $C\in\mathcal{T}^s$,
  then
  \begin{align*}
  \mathbf{Rw}(A) & \;=\; \max \{\mathbf{Rw}(B),\mathbf{Rw}(C)-1\}               \\
                 & \;=\; \max \{\mathbf{h}(\Psi(B)),\mathbf{h}(\Psi(C))-1\}     \\
                 & \;=\; \max\{ \mathbf{Height}(a,b), (a,b)\in \Psi(B)\Psi(C)^r \}
  ,
  \end{align*}
   which is equal to $\mathbf{Height}(\Psi(A))$. If $A$  is factorized as $(v)$ with $B\in\mathcal{T}^s$ and $C\in\mathcal{T}\backslash\mathcal{T}^s$,
   then $C$ can be factorized as $(v)$ with
$D\in\mathcal{T}^s$ and $E\in\mathcal{T}$, and using an induction we have:
  \begin{align*}
  \mathbf{Rw}(A) & \;=\; \max\{\mathbf{h}(\Psi(B)),\mathbf{h}(\Psi(D))-1,\mathbf{Height}(\Psi(E))-2\}   \\
                 & \;=\; \max\{ \mathbf{Height}(a,b), (a,b)\in \Psi(B)\Psi(D)^r \Psi(E)\}
  ,
  \end{align*}
which gives exactly
$\mathbf{Height}(\Psi(A))$.

 When $A\in\mathcal{T}^s$, it  satisfies  $(i)$, $(ii)$, $(iii)$ or $(iv)$, and the recursive definition of $\Psi$ implies that
 $\mathbf{Nbp}(A,i)=\mathbf{Nbu}(\Psi(A),i)$. Otherwise, if $A$  is factorized as $(v)$ with $B\in\mathcal{T}^s$ and $C\in\mathcal{T}^s$,
  then $\mathbf{Nbp}(A,i)=\mathbf{Nbp}(B,i)+\mathbf{Nbp}(C,i+1)$, and using the recurrence hypothesis it is equal to
  \[
  \mathbf{Nbu}(\Psi(B),i)+ \mathbf{ Nbu}(\Psi(C),i+1)=\mathbf{Nbu}(\Psi(B)\Psi(C)^r,i)=\mathbf{Nbu}(\Psi(A),i).
  \]
   Whenever $A$ is factorized as $(v)$ with $D\in\mathcal{T}^s$ and $E\in\mathcal{T}$, a similar argument completes the proof.
\end{proof}

Due to the symmetry $\sigma$ about the diagonal $y=x$, the two statistics {\bf Lw}$(\cdot)+1$ and {\bf Rw}$(\cdot)$ have the same distribution on directed animals in $\mathcal{T}$
and $\mathcal{Q}$. Using Theorem~\ref{thmStatistics} and Table~\ref{tab2}, this induces that {\bf Cross}$(\cdot)+1$ and {\bf Height}$(\cdot)$ also have the same distribution
 in $\mathcal{GD}$ and $\mathcal{GD}^{\star}$.


\begin{thebibliography}{10}

\bibitem{Alb} M. Albenque.
\newblock A note on the enumeration of directed animals via gas considerations.
\newblock {\em Ann. Appl. Probab.}, 19(5), 1860-1879, 2009.

\bibitem{Bach} A. Bacher.
\newblock Average site perimeter of directed animals on the two-dimensional lattices.
\newblock {\em Discrete Mathematics}, 312(5), 1038--1058, 2012.

\bibitem{Band} C. Banderier, C. Krattenthaler, A. Krinik, D. Kruchinin, V. Kruchinin, D. Nguyen, and M. Wallner.
\newblock Explicit Formulas for Enumeration of Lattice Paths: Basketball and the Kernel Method.
\newblock In G.E. Andrews, C. Krattenthaler, and A. Krinik (eds), {\em Lattice Path Combinatorics and Applications}, 78--118, Springer, 2019.

\bibitem{Barc} E. Barcucci, A. Del Lungo, E. Pergola,  and R. Pinzani.
\newblock Directed animals, forests and permutations.
\newblock {\em Discrete Mathematics}, 204, 41--71, 1999.

\bibitem{Betr} J. Betrema, and J.G. Penaud.
\newblock Animaux et arbres guingois.
\newblock {\em Theoretical Computer Science}, 117(1-2), 67--89, 1993.

\bibitem{Bousq} M. Bousquet-M{\'e}lou.
\newblock Une bijection entre les polyominos convexes dirig{\'e}s et les mots de Dyck bilat{\`e}res.
\newblock {\em RAIRO - Theoretical Informatics and Applications - Informatique Théorique et Applications}, 26(3), 205--219, 1992.

\bibitem{Bou1} M. Bousquet-M{\'e}lou.
\newblock New enumerative results on two-dimensional directed animals.
\newblock {\em Discrete Mathematics}, 180(1-3), 73--106, 1998.

\bibitem{Bou} M. Bousquet-M{\'e}lou, and A. Rechnitzer.
\newblock Lattice animals and heaps of dimers.
\newblock {\em Discrete Mathematics}, 258, 235--274, 2002.

\bibitem{Broa} S.R. Broadbent, and J.M. Hammersley.
\newblock Percolation processes. I. Crystals and mazes.
\newblock {\em Proc. Camb. Phil. Soc.}, 53, 629--641, 1957.

\bibitem{CoGu}  A.R.  Conway, and  A.J.  Guttmann.
\newblock  On  two-dimensional  percolation.
\newblock  {\em J.  Phys.  A:  Math. Gen.}, 28, 891--904, 1995.

\bibitem{Conw}  A.R. Conway.
\newblock Some exact results for moments of 2D-directed animals.
\newblock {\em J. Phys. A: Mathematical and General}, 29(17), 5273--5283, 1996.

\bibitem{Dha} D. Dhar.
\newblock Equivalence of two-dimensional directed-site animal problem to Baxter's hard square lattice gas model.
\newblock {\em Phys. Rev. Lett.}, 49, 959--962, 1982.

\bibitem{Dhar} D. Dhar, M.K. Phani, and M. Barma.
\newblock Enumeration of directed site animals on two-dimensional lattices.
\newblock {\em J. Phys. A: Math. Gen.}, 15, 279--284, 1982.

\bibitem{FS} P. Flajolet, and R. Sedgewick.
\newblock Analytic combinatorics. Cambridge University Press, 2009.

\bibitem{Gou} D. Gouyou-Beauchamps, and G. Viennot.
\newblock Equivalence of two-dimensional directed animal problem to a one-dimensional path problem.
\newblock {\em Advances in Appl. Math.}, 9, 334--357, 1988.

\bibitem{Har} F. Harary.
\newblock Graph Theory and Theoretical Physics.
\newblock Ed. Harary, Academic Press, New York, 1967.

\bibitem{Knu} D.E. Knuth.
\newblock The art of computer programming. Vol. 4, Addison-Wesley, Reading MA, 1973, Third edition, 1997.

\bibitem{Leb} Y. Le Borgne, and J.F. Marckert.
\newblock Directed animals and gas models revisited.
\newblock {\em The Electronic J. of Combinatorics}, 14, \# R71, 2007.

\bibitem{Marck} J.F. Marckert.
\newblock Directed animals, quadratic systems and rewriting systems.
\newblock {\em The Electronic J. of Combinatorics}, 19(3), \# P45, 2012.

\bibitem{Mer} D. Merlini, R. Sprugnoli, and M. Verri.
\newblock Some statistics on Dyck paths.
\newblock {\em J. Statist. Plann. Inference}, 101, 211--227, 2002.

\bibitem{Sap} A. Sapanoukis, I. Tasoulas and P. Tsikouras.
\newblock Counting strings in Dyck paths.
\newblock {\em Discrete Mathematics}, 307(23), 2909--2924, 2007.

\bibitem{Slo} N.J.A. Sloane.
\newblock On-Line Encyclopedia of Integer Sequences.
\newblock Published electronically at http://oeis.org/.

\bibitem{Stan} R. Stanley.
\newblock Enumerative Combinatorics, Vol. 1. Cambridge University Press, 1997.

\bibitem{Stan2} R. Stanley.
\newblock Enumerative Combinatorics, Vol. 2. Cambridge University Press, 1999.

\bibitem{Sun} Y. Sun.
\newblock The statistic ``number of udu's'' in Dyck paths.
\newblock {\em Discrete Mathematics}, 287, 177--186, 2004.

\bibitem{Vie} G.X. Viennot.
\newblock Heaps of pieces.
\newblock {\em I. Basic definitions and combinatorial lemmas}, In Combinatoire {\'e}num{\'e}rative (Montr{\'e}al, Que., 1985), Vol. 1234 de Lecture notes in Math., 231--350, Springer, Berlin, 1986.

\bibitem{Vie2} G.X. Viennot.
\newblock Multi-directed animals, connected heaps of dimers and Lorentzian triangulations.
\newblock {\em Journal of Physics: Conference Series}, 42, 268--280, 2006.
\end{thebibliography}
\end{document}